\documentclass[ preprint
              , 1p
              , times
              , sort&compress
              ]{elsarticle}
\pdfoutput=1

\usepackage[T1]{fontenc}
\title{Accessibility and presentability in 2-categories}
\makeatletter
\@namedef{u8:\detokenize{ť}}{\ooalign{\kern.1em\raisebox{.2em}{\scalebox{.75}{'}}\cr t}}
\makeatother

\journal{JPAA}
\def\definetac{\newif\iftac}

\ifx%
	\tactrue\undefined%
	\definetac%
	\ifx%
		\state\undefined\tacfalse%
	\else%
		\tactrue%
	\fi%
\fi%

\iftac\else\usepackage{amsthm}\fi

\usepackage{ amssymb
	, array
	, amsmath
	, stmaryrd
	, mathrsfs
	, enumitem
	, mathtools
	, tikz
	, url
	, booktabs
	, xspace
	, xcolor
}

\usepackage{tikz-cd}

\definecolor{darkgreen}{rgb}{0,0.45,0}
\definecolor{lightgray}{gray}{0.45}

\usepackage[all,2cell,pdf]{xy}\UseAllTwocells

\usepackage[ hyperfootnotes=true
	, colorlinks=true%
	, citecolor=black
	, linkcolor=black
	, filecolor=black
	, urlcolor=black
	, pdftitle={Accessibility and presentability in 2-categories}
]{hyperref}

\usepackage[color=gray!20]{todonotes}

\usepackage[utf8]{inputenc}
\usetikzlibrary{decorations.pathmorphing}
\usepackage[makeroom]{cancel}
\usepackage{xparse}

\ExplSyntaxOn
\NewDocumentCommand{\makeabbrev}{mmm}
{
	\yoruk_makeabbrev:nnn { #1 } { #2 } { #3 }
}

\cs_new_protected:Npn \yoruk_makeabbrev:nnn #1 #2 #3
{
	\clist_map_inline:nn { #3 }
	{
		\cs_new_protected:cpn { #2 } { #1 { ##1 } }
	}
}
\ExplSyntaxOff

\makeabbrev{\mathbf}{b#1}{b,c,d,e,g,h,i,j,k,l,m,n,o,p,q,r,t,u,v,w,x,y,z,%
	B,C,D,E,G,H,I,J,K,L,M,N,O,P,Q,R,T,U,V,W,X,Y,Z}

\makeabbrev{\boldsymbol}{bs#1}{%
	a,b,c,d,e,g,h,i,j,k,l,m,n,o,p,q,r,s,t,u,v,w,x,y,z,%
	A,B,C,D,E,G,H,I,J,K,L,M,N,O,P,Q,R,S,T,U,V,W,X,Y,Z}

\makeabbrev{\mathsf}{sf#1}{a,b,c,d,e,f,g,h,i,j,k,l,m,o,p,q,r,s,t,u,v,w,x,y,z,%
	A,B,C,D,E,F,G,H,I,J,K,L,M,N,O,P,Q,R,S,T,U,V,W,X,Y,Z}
\makeabbrev{\mathfrak}{f#1}{a,b,c,d,e,f,g,h,j,k,l,m,n,o,p,q,r,s,t,u,v,w,x,y,z,%
	A,B,C,D,E,F,G,H,I,J,K,L,M,N,O,P,Q,R,S,T,U,V,W,X,Y,Z}
\makeabbrev{\mathbf}{l#1}{A,B,C,D,E,F,G,H,I,J,K,L,M,N,O,P,Q,R,S,T,U,V,W,X,Y,Z,%
	a,b,c,d,f,h,i,j,k,m,o,p,r,s,t,u,v,w,x,y,z}

\makeabbrev{\mathcal}{c#1}{A,B,C,D,E,F,G,H,I,J,K,L,M,N,O,P,Q,R,S,T,U,V,W,X,Y,Z}
\makeabbrev{\mathbb}{s#1}{A,B,C,D,E,F,G,H,I,J,K,L,M,N,O,P,Q,R,S,T,U,V,W,X,Y,Z}

\makeatletter
\newif\ifhyperref
\@ifpackageloaded{hyperref}{\hyperreftrue}{\hyperreffalse}
\iftac
	\let\your@state\state
	\def\state#1{\gdef\currthmtype{#1}\your@state{#1}}
	\let\your@staterm\staterm
	\def\staterm#1{\gdef\currthmtype{#1}\your@staterm{#1}}
	\let\defthm\newtheorem
	\def\currthmtype{}
	\ifhyperref
		\def\autoref#1{\ref*{label@name@#1}~\ref{#1}}
	\else
		\def\autoref#1{\ref{label@name@#1}~\ref{#1}}
	\fi
	\AtBeginDocument{%
		\let\old@label\label%
		\def\label#1{%
			{\let\your@currentlabel\@currentlabel%
					\edef\@currentlabel{\currthmtype}%
					\old@label{label@name@#1}}%
			\old@label{#1}}
	}
\else
	\ifhyperref
		\def\defthm#1#2{%
			\newtheorem{#1}{#2}[section]%
			\expandafter\def\csname #1autorefname\endcsname{#2}%
			\expandafter\let\csname c@#1\endcsname\c@thm}
	\else
		\def\defthm#1#2{\newtheorem{#1}[thm]{#2}}
		\ifx\SK@label\undefined\let\SK@label\label\fi
		\let\old@label\label
		\let\your@thm\@thm
		\def\@thm#1#2#3{\gdef\currthmtype{#3}\your@thm{#1}{#2}{#3}}
		\def\currthmtype{}
		\def\label#1{{\let\your@currentlabel\@currentlabel\def\@currentlabel%
					{\currthmtype~\your@currentlabel}%
					\SK@label{#1@}}\old@label{#1}}
		\def\autoref#1{\ref{#1@}}
	\fi
\fi
\makeatother

\DeclareDocumentEnvironment{enumtag}{m }{\begin{enumerate}[label=\textsc{#1}\oldstylenums{\arabic*}),ref=\textsc{#1}\oldstylenums{\arabic*}]}{\end{enumerate}}

\newsavebox{\fminipagebox}
\NewDocumentEnvironment{fminipage}{m O{\fboxsep}}
{\par\kern#2\noindent\begin{lrbox}{\fminipagebox}
		\begin{minipage}{#1}\ignorespaces}
			{\end{minipage}\end{lrbox}%
	\makebox[#1]{%
		\kern\dimexpr-\fboxsep-\fboxrule\relax
		\fbox{\usebox{\fminipagebox}}%
		\kern\dimexpr-\fboxsep-\fboxrule\relax
	}\par\kern#2
}

\numberwithin{equation}{section}
\iftac\theoremstyle{plain}\else\theoremstyle{definition}\fi
\defthm{thm}{Theorem}
\defthm{cor}{Corollary}
\defthm{prop}{Proposition}
\defthm{lem}{Lemma}
\defthm{sch}{Scholium}
\defthm{assume}{Assumption}
\defthm{claim}{Claim}
\defthm{conj}{Conjecture}
\defthm{hyp}{Hypothesis}
\defthm{fact}{Fact}
\defthm{quest}{Question}
\defthm{defn}{Definition}
\defthm{defn-prop}{Definition-Proposition}
\defthm{notat}{Notation}
\defthm{rmk}{Remark}
\defthm{eg}{Example}
\defthm{egs}{Examples}
\defthm{ex}{Exercise}
\defthm{ceg}{Counterexample}
\defthm{con}{Construction}
\defthm{warn}{Warning}
\defthm{digr}{Digression}
\defthm{per}{Perspective}
\defthm{disc}{Discussion}
\defthm{term}{Terminology}
\defthm{heur}{Heuristics}
\iftac\theoremstyle{plain}\else\theoremstyle{remark}\fi

\numberwithin{equation}{section}
\newtheorem*{thm*}{Theorem}
\newtheorem*{notat*}{Notation}
\DeclareSymbolFont{bbold}{U}{bbold}{m}{n}

\DeclareMathOperator*{\colim}{colim}

\let\xto\xrightarrow

\providecommand{\Nearrow}{\rotatebox[origin=c]{45}{$\Rightarrow$}}

\providecommand{\Swarrow}{\rotatebox[origin=c]{225}{$\Rightarrow$}}

\def\twoc#1{\mathsf{#1}}
\def\CAT{\twoc{CAT}}
\def\Cat{\twoc{Cat}}

\def\Lex{\twoc{Lex}}
\def\Rex{\twoc{Rex}}
\def\LFP{\twoc{LFP}}
\def\LP{\twoc{LP}}

\def\Pos{\twoc{Pos}}

\def\Set{\one{Set}}
\def\sSet{\one{sSet}}

\def\op{\mathrm{op}}

\def\coop{\mathrm{coop}}

\def\To{\Rightarrow}

\newcommand{\Nat}{\mathrm{Nat}}

\newlength{\seplen}
\newlength{\sepwid}
\def\firstblank{\,\rule{\seplen}{\sepwid}\,}

\usepackage{CJKutf8}
\newcommand{\japa}[1]{\text{\begin{CJK}{UTF8}{min}#1\end{CJK}}}
\newcommand{\yon}{\japa{よ}}
\usepackage{turnstile}

\def\lan{\mathrm{lan}}
\def\Lan{\mathrm{Lan}}

\renewcommand{\textbf}[1]{\text{\fontseries{b}\selectfont{\upshape #1}}}

\def\yomod{\upsilon\text{-}{\sf Mod}}
\def\yoth{\upsilon\text{-}{\sf Th}}

\newcommand{\gray}[1]{{\color{lightgray} #1}}

\def\id{\text{id}}

\def\kz{KZ doctrine\xspace}
\def\kzs{{\kz}s\xspace}

\def\bsInd{\boldsymbol{I\kern-.1em nd}}
\let\circ\cdot

\def\gu{\textsc{gu}\xspace}

\usepackage[normalem]{ulem}
\usepackage{contour}

\contourlength{0.8pt}

\newcommand{\one}[1]{\mathrm{#1}}

\def\yc{context\xspace}
\def\hk{\hookrightarrow}

\def\El{\cE}

\usepackage{hyphenat}
\setlist[1]{itemsep=0pt}

\def\[{\begin{equation}}
			\def\]{\end{equation}}

\usepackage{ulem, xcolor,soul}

\sethlcolor{green}

\allowdisplaybreaks

\author[1]{Ivan Di Liberti\corref{cor1}}
\ead{diliberti.math@gmail.com}
\author[2]{Fosco Loregian}
\ead{fosco.loregian@taltech.ee}

\cortext[cor1]{Corresponding author}

\affiliation[1]{%
  organization = {Institute of Mathematics, Czech Academy of Sciences},
  addressline  = {\v{Z}itn\'a 25},
  postcode     = {115 67},
  city         = {Prague},
  country      = {Czech Republic}%
}
\affiliation[2]{
  organization = {Tallinn University of Technology},
  addressline  = {Akadeemia tee 21b},
  postcode     = {12618},
  city         = {Tallinn},
  country      = {Estonia}%
}

\begin{document}
\begin{abstract}
	We outline a definition of accessible and presentable objects in a 2-category $\cK$ endowed with a ``KZ context'', that is to say a pair of lax-idempotent monads interacting in a prescribed way; this perspective suggests a unified treatment of many ``Gabriel-Ulmer like'' theorems, asserting how presentable objects arise as reflections of generating ones. We outline the notion of \emph{(Gabriel-Ulmer) envelope} for a KZ context, sufficient to concoct Gabriel-Ulmer duality. We end the paper with a roundup of examples, involving classical (set-based and enriched), low dimensional category theory, and a perspective for future work, rooted in higher category theory and homotopy theory.
\end{abstract}

\maketitle

\section{Introduction}
\begin{quote}
	``\textit{The theory of categories enriched in some base closed category $\cV$, is couched in set-theory; some of the interesting results even require a hierarchy of set-theories. Yet, there is a sense in which the results themselves are of an elementary nature.}'' \hspace*{\fill} \cite{StreetCosmoi1974}
\end{quote}
The theory of accessible and presentable categories has nowadays gained a primary position in categorical algebra due to its connections with categorical model theory \cite{makkai1989accessible}, homotopy theory \cite{Rosicky2009}, universal algebra \cite{adamek2011algebraic}; the definition of locally presentable category is currently considered deeply rooted in set theory, and rightly so.

Intuitively speaking, the notions of accessibility and local presentability are controlled by a regular cardinal $\lambda$: both definitions involve $\lambda$-filtered colimits, the notion of $\lambda$-presentable object, and the presence of a small set of $\lambda$-presentable objects generating the category under $\lambda$-filtered colimits.

The notion of regular cardinal is intrinsically set-theoretical and goes against the grain for most category theorists: yet, we claim that a more intrinsic categorical treatment of accessibility and presentability is possible and that it would help to put the abstract theory of accessibility and presentability in a broader perspective.

The present work stems from the desire to determine precisely to which extent this claim can be attained and what, in the definitions of accessibility and presentability, can be reduced to the implant of \emph{formal} category theory. This would provide a unifying language to make accessibility and presentability two notions that stem from the expressiveness of the language of a 2-category, and a more intrinsic understanding of the many applications of locally presentable categories.

Taking the above incipit of \cite{StreetCosmoi1974}  as an inspiration, we are guided by the following questions:
\begin{enumtag}{q}
	\item \label{q:zero} \emph{To which extent is it possible to outline the formal content underlying the definition of accessibility and presentability for the objects of an abstract 2-category $\cK$?} It turns out that the gist of the definition can be made quite independent from the set-theoretic background, paying a reasonable price: accessible objects can be recognised as those in the image of a suitable \emph{lax idempotent 2-monad} $\bsS$ (a ``KZ doctrine'', or with less dramatic of an acronym, just a \emph{doctrine}) on $\cK$.

	The classical theory can be recovered if the doctrine is $\bsInd_\lambda$, namely the free cocompletion under $\lambda$-filtered colimits. The price to pay is that we cannot characterise accessibility, but only `$\lambda$-accessibility', or $\bsS$-relative accessibility.\footnote{
		$\bsS$ can in principle play the r\^ole of any cocompletion operation whatsoever, for example with respect to a generic \emph{sound doctrine} $\sD$ in the sense of \cite[2.2]{adamek2002classification}; our result frames the fact that given such a sound doctrine $\sD$, a category $A$ is $\sD$-\emph{accessible} if and only if $A\cong \sD\text{-Ind}(S)$ for some small $S$.
	}
	\item \label{q:due} \emph{Can presentable objects internal to a 2-category $\cK$ be characterised as those which are both accessible and cocomplete?} Showing that this is true is the first original result in the present work: cocompleteness in a 2-category must, however, again be specified \emph{relative} to some additional structure; for this purpose, we fix a second doctrine $\bsP$, and we define `$\bsP$-cocompleteness' through the existence of certain left extensions, as in \autoref{cocomplete_obj}, drawing from previous work of C. Walker \cite{walker}.

	The paradigmatic example of such $\bsP$ is the (covariant) presheaf construction, where $\bsP A$ is the category of functors $A^\op\to\Set$.
\end{enumtag}
In short, our approach to the questions above can be summarised as follows.
\begin{quote}
	Let $\cK$ be a 2-category whose objects we think of as locally small categories of some kind; then, accessibility and presentability are defined in terms of the interplay between two doctrines $\bsS,\bsP$, encoded in terms of a natural map $\upsilon : \bsS\To \bsP$: in simple terms, an object $A$ of a 2-category $\cK$ is ($\upsilon$-)accessible if it is of the form $\bsS G$ for some $G$, and it is ($\upsilon$-)presentable if it is a reflective subobject of some $\bsP X$ for some $X$, and $i$ ``plays well with $\bsS$-colimits'', in the sense specified by our \autoref{def_context}.
\end{quote}
This notation is chosen in order to make it evident the classical setting from which we abstract: a category is accessible if and only if it is the $\text{Ind}_\lambda$-completion $\bsInd_\lambda(G)$ of some small category $G$; a category is locally presentable if and only if it is a reflection of some category of presheaves, and the embedding $X \hookrightarrow \bsP A$ preserves filtered colimits.

Thus, a convincing `formal theory of accessibility' for $\cK$ shall provide a convincing explanation and abstraction of the fact that presentable objects in various 2-categories arise as \emph{reflective localisations} of presheaf objects. In this sense, a reasonable soundness request for an abstract theory of accessibility is that the \emph{Gabriel\hyp{}Ulmer representation theorem} \cite{Gabriel1971} (see \cite{centazzo2004generalised} for a general and enlightening perspective on the process of ``completion under a fixed class of shapes'') shall hold when $\Set^{A^\op}$ is substituted by an object of the form $\bsP A$ in $\cK$.

We define a \emph{Gabriel\hyp{}Ulmer envelope} in \autoref{guenvelope} as a 1-cell $\iota_A : A\to\bsD A$ such that $\bsD A \to \bsS(\bsD A)$ exhibits a suitable universal property; in particular $\bsD$ is another doctrine. Given such an envelope, we can define a bi-equivalence of 2-categories
\[
	\yomod : \upsilon\text{-}\Rex^\coop \rightleftarrows \upsilon\text{-}\LP : \yoth
\]
between the 2-category $\LP(\yon)$ of $\yon$-presentable objects in $\cK$, and certain reflective localisations of objects of the form $\bsD G$, collected in a 2-category $\Rex(\yon)$.

We explore the consequences of this definition in our §\ref{sec:gabriel}.

\subsection{Organisation of the paper}
Section 2 starts with an introductory section aimed at establishing notation; all the material in §2.1 is attributed to some reference or likely to be considered folklore among 2-category theorists; however, it is somewhat difficult to find a single comprehensive reference containing explicit mention of all the technical results we are going to need. Thus, section 2 aims to build an easy-to-cite reference, as comprehensive as possible, for the fundamental results on doctrines.

The central notion of \emph{admissibility} with respect to a doctrine is the main topic of §2.2; this classical notion is the building block for the original material that follows: our \autoref{def_context}, that is in its own right the main building block for the whole paper. Finally, §\ref{sec_acc_pres} introduces the notions of accessibility and presentability relative to a context; the representation theorem characterising presentable objects as accessible and cocomplete is \autoref{the-main}.

Section 3 is aimed to prove Gabriel-Ulmer duality in an abstract context; we introduce the notion of \emph{Gabriel-Ulmer envelope} relative to a context $\upsilon : \bsS \To \bsP$ and sketch its relation with a notion of Cauchy-completeness in §\ref{morita_in_ctx}. The notion of \gu envelope is, at the best of our knowledge, an original contribution, and we outline its fundamental r\^ole in the proof of Gabriel-Ulmer duality in §\ref{gu_in_ctx}: the main theorem is \autoref{guduality}.

Section 4 studies examples: classical presentable categories in $\CAT$, $\sD$-presentability with respect to a sound doctrine, posets regarded as $\{0,1\}$-enriched categories (cf. \cite{porst2011algebraic}), the general case of enriched categories. We end the paper with a subsection of conjectures, aimed at future development.

\section*{Acknowledgments}

The present paper was initially written when both authors were appointed at Masaryk University, Brno; it was almost completely rewritten years after, to put it in its current form. The first author is currently supported by the Grant Agency of the Czech Republic project EXPRO 20-31529X and RVO: 67985840. The second author is currently supported by the ESF funded Estonian IT Academy research measure (project 2014-2020.4.05.19-0001).
  Both authors are immensely grateful to the anonymous referee of this paper. A first version of this paper contained a couple of faulty arguments and needed extensive polishing; the anonymous referee provided invaluable insight in this sense. Most prominently, they have suggested how to polish \autoref{def_context} and \autoref{caucau} (a notion called \emph{climbability} in an earlier draft) and they have donated \autoref{moritaleftadj}, which fixed a gap in a key result, \autoref{production}.
  The second author owes a special thanks to Marko Halanevych, Iryna Kovalenko, Olena Tsybulska, Nina Garenetska and Vladyslav Troitskyi.
\section{Accessibily and Presentability}
The present section contains the main definitions of the paper, and it is structured in subsections as follows.
\begin{enumtag}{i}
	\item We recap the definition of KZ doctrine in \autoref{defkzdoctrine}, recalling and describing the key examples for our interests.
	\item We discuss the notion of \emph{admissible arrow} in \autoref{defadm} for a given doctrine. This notion is less well-known and is closely related to the `nerve-realisation paradigm' relative to a specified doctrine.
	\item We introduce the notion of KZ context in \autoref{def_context}. This is the technology on which our definition of accessibility and presentability is based. In a nutshell, it amounts to a couple $(\bsS, \bsP)$ of doctrines interacting together.
	\item We define accessible and presentable objects for a KZ context in \autoref{yonacc} and \autoref{yonpres}, and we recover a classical result of the theory of locally presentable categories.
\end{enumtag}
\subsection{A brief introduction to \kzs}
Our go-to reference will be Walker's paper \cite{walker}; although it might be debatable that this is a `standard' reference, it has the merit of being very concise and crafted in the same direction as the present work.

The reader interested in a more wide-ranging introduction to the topic of \kzs (so-called because they were introduced by A. Kock \cite{kock1995monads} and V. Z\"oberlein \cite{zobbe}) or more modern presentations in \cite{marmolejo1997doctrines}, where the notion is given internally to an arbitrary \textbf{Gray}-category.

Contrary to this last very terse presentation, we will aim at the reader unfamiliar with 2-dimensional monad theory, commenting on the definitions along the way with some remarks and examples.
\begin{defn}[\kz, Definition 2 of \cite{walker}]
	\label{defkzdoctrine}
	A \emph{\kz}, or simply a \emph{doctrine} $(\bsS,\sigma)$ on a 2-category $\cK$ consists of
	\begin{enumtag}{kz}
		\item \label{kz_1} A function on objects, $\bsS_0 : \cK_0\to\cK_0$;
		\item \label{kz_2} For every object $ A\in\cK$, a 1-cell $\sigma_A\colon A\to \bsS A$;
		\item \label{kz_3} For every pair of objects $ A\text{ and } B$ and 1-cell $f\colon A\to \bsS B$, a left extension
		\[
			\vcenter{\xymatrix{
					\bsS A\ar[r]^{\bar f} & \dltwocell<\omit>{<2>c_f} \bsS B\\
					A\ar[ur]_{f} \ar[u]^{\sigma_A}&
				}}
			\label{d21}
		\]
		of $f$ along $\sigma_A$ exhibited by a 1-cell $\bar f : \bsS A \to \bsS B$ and an isomorphism $c_f : f \To \bar f \cdot \sigma_A$.
	\end{enumtag}
	Moreover, we require that:
	\begin{enumtag}{kp}
		\item \label{kp_1}  For every object $ A\in\cK$, the left extension
		of $\sigma_A$ as in \eqref{d21} is given by the triangle
		\[
			\vcenter{\xymatrix{
			\bsS A\ar@{=}[r]^{\id _{\bsS A}} & \bsS A\\
			&  A\ar[u]_{\sigma_A} \ar[ul]^{\sigma_A}
			}}
		\]
		filled by the identity 2-cell of $\sigma_A$. Note that this means $c_{\sigma_A}$ is equal to the identity
		2-cell on $\sigma_A$.

		\item \label{kp_2} For any 1-cell $g\colon B\to \bsS C $, the corresponding
		left extension $\bar{g}\colon \bsS B\to \bsS C $
		preserves the left extension $\bar{f}$ in \eqref{d21}. Recall that this means that in thediagram below, the $2$-cell exibits $\bar{g}\bar{f}$ as the left Kan extension of $\bar{g}f$ along $\sigma_A$.

			\[
		\vcenter{\xymatrix{
				\bsS B\ar[r]^{\bar g} & \bsS C \dltwocell<\omit>{<2>c_g} & \bsS A \ar[r]^{\bar f} & \bsS B\dltwocell<\omit>{<2>c_f}\ar[r]^{\bar g} & \bsS C \\
				B \ar[u]^{\sigma_B}\ar[ur]_{g} && A\ar[u]^{\sigma_A} \ar[ur]_f
			}}
	\]
	\end{enumtag}
\end{defn}

We will often invoke the following two notions of \emph{representably fully faithful} 1-cell and of \emph{pseudo monomorphism}.
\begin{defn}\label{fufai}
	A 1-cell $i : A \to B$ of a 2-category $\cK$ is called \emph{(representably) fully faithful} if for every object $X\in\cK$ the functor
	\[ i_* : \cK(X,A) \to \cK(X,B) \]
	obtained by post-composition with $i$ is fully faithful.
\end{defn}

\begin{rmk}\label{pseudomono}
	A representably fully faithful $1$-cell is clearly a \emph{pseudo monomorphism} in the following sense: given 1-cells $f,g : E\to A$ and $i : A \to B$ as above, every invertible 2-cell $\alpha : if\To ig : E\to B$ comes from an invertible 2-cell $\bar\alpha : f \To g : E \to A$.
\end{rmk}

\begin{rmk}[Doctrines and lax-idempotent pseudomonads]
	It is well-known \cite{MOGGI199155} that there is a bijective correspondence between monads $T$ on a category $\cC$, equipped with a pair of transformations $(\mu,\eta)$ (multiplication and unit), and pointed endofunctors $T : \cC \to \cC$ equipped with \emph{Kleisli extension} maps
	\[(\firstblank)^* : (f :A\to TB)\mapsto (f^* : TA \to TB).\]
	\autoref{defkzdoctrine} above specifies a structure of the second kind; but every doctrine on $\cK$ can be extended to a (pseudo)monad $(\bsS, \sigma, \mu)$ on $\cK$, and this correspondence remains 1-to-1. More than often thus we treat a doctrine as if it was the associated monad.

	The pseudomonads arising from this construction satisfy an additional property thanks to \ref{kp_1}, \ref{kp_2}: these are called \emph{lax-idempotent} pseudomonads. \cite[Remark 3]{walker} outlines this connection shortly; for our purposes it will be sufficient to describe the functoriality of the pseudomonad associated to a doctrine: given a $1$-cell $f: A \to B$, we call $\bsS_! f$ the $1$-cell below, which exists by \ref{kz_3}:
	\[
		\vcenter{\xymatrix{
		A\ar[d]_f \ar[r]^{\sigma_A} & \bsS A \ar@{.>}[d]^{\bsS_! f := \lan_{\sigma_A}(\sigma_Bf)} \\
		B \ar[r]_{\sigma_B} & \bsS B\ultwocell<\omit>{}
		}}
	\]
\end{rmk}
The universal property of the left extension in question yields that $\bsS_!$ is a functor between hom-categories $\cK(A,B)$ and $\cK(\bsS A, \bsS B)$.
\begin{eg}[Small presheaves and small colimits]\label{smallpre}
	The paradigmatic example of a doctrine is the construction of \emph{small presheaves} $(\bsP, \yon)$, on the $2$-category $\cK=\CAT$ of locally small categories. The functor $\bsP$ is defined on objects mapping a category $A$ to its category of small presheaves $\bsP A$ \cite{ROSICKY1999261,day2007limits}, namely to the locally small full subcategory of $\Set^{A^\op}$, which is not locally small, spanned by those presheaves that result from small colimits of representables.

	The unit of this monad is the Yoneda embedding $\yon_A : A \to \bsP A$ that exists since $A$ is locally small, while the doctrine axioms amount to the fact that $\bsP A$ is the free completion of $A$ under (small) colimits.

	Because of the importance of this doctrine for our paper, it is worth describe in detail the action of the Kleisli extension. Given a diagram
	\[
		\vcenter{\xymatrix{
		A \ar@{}[dr]|(.3){c_f\Swarrow} \ar[d]_{\yon_A} \ar[r]^f & \bsP B \\
		\bsP A \ar[ur]_{\lan_{\yon_A} f} &
		}}
	\]
  the extension $\lan_{\yon_A}(f)(p)$ acts on objects as the following colimit, where $\pi_P: \El(P) \to A$ is the projection from the category of elements of a presheaf $P \in \bsP A$:
	\[\lan_{\yon_A}(f)(P) = \colim (\El(P) \xto{\pi_P} A \xto{f} \bsP B).\]
	Notice that $\colim \pi_P$ exists and can be computed because $P$ is a small presheaf (and thus there exist a small category $C$ and a cofinal functor $C \to \El(P)$) and because $\bsP B$ is small-cocomplete.

\end{eg}
\begin{rmk}[Size issues]
	This example unveils the foundational background on which this paper is built. In most of the concrete circumstances, the collection of objects of our $\cK$ will a be a conglomerate (an `illegitimate class'), and each hom category $\cK(X,Y)$ is locally small. Of course, this could be phrased in terms of Grothendieck universes, leading to no relevant changes.
\end{rmk}
\begin{notat}\label{lr_shriek}
	Along the paper the unit of a doctrine will be denoted with the Greek letter corresponding to its name (so, $\sigma$ in the case of a doctrine $\bsS$). This convention comes in handy to avoid spelling out explicitly which unit is which in most of our proofs, but we stipulate that the small presheaves construction makes an exception, since we instead employ the hiragana `yo' $\yon$ to denote its unit to stress the fact that $\yon_A : A \to \bsP A$ plays the r\^ole of the Yoneda embedding.
\end{notat}
\begin{eg}[Ind-completion and filtered colimits]\label{bsind}
	Another classical example of doctrine is the \emph{Ind-completion} (cf. \cite[Ch. 6]{KS2}) of a category. For a given regular cardinal $\lambda$, we define $\bsInd_\lambda(A)$ as the full subcategory of $\bsP A$ spanned by \emph{$\lambda$-flat} \cite[§6.3]{Bor1} functors. The unit is still the Yoneda embedding, since every representable presheaf is $\lambda$-flat, whichever cardinal $\lambda$ is chosen.

	Again, the doctrine axioms boil down to the fact that the Ind-completion is a free completion operation, this time under $\lambda$-\emph{directed} colimits \cite[Thm. 2.26]{Adamek1994}. This doctrine is again naturally defined on the $2$-category of locally small categories.
\end{eg}
\begin{eg}[Free completion under $\lambda$-small colimits]\label{finitecolimits}
	Another example that will come in handy is the free completion under \emph{$\lambda$-small colimits}. We denote this doctrine $(\bsD_\lambda, \delta)$ or, in case $\lambda=\omega$, just with $\bsD$. $\bsD_\lambda(A)$ can be described as the full subcategory of $\bsP A$ of those presheaves that are $\lambda$-small colimits of representables. This doctrine, as well as other examples, is discussed by \cite[5.9]{Kelly2005}.
\end{eg}
\begin{eg}[Cauchy completion]
	Another important example that we will heavily use in \autoref{sec:gabriel} is the \emph{Cauchy completion} $\hat A$ (also known as Karoubi completion, `completion under absolute colimits') of a category (cf. \cite[§6.5]{Bor1} and \cite{CTGDC_1986__27_2_133_0}). In $\CAT$, the Cauchy completion of a category coincides with the splitting of idempotents, or with the closure of the image of the Yoneda embedding under retracts (which also coincides with the full subcategory of tiny objects, cf. \cite{CTGDC_1986__27_2_133_0}). This presentation for $\hat A$ yields a factorisation of the Yoneda embedding as $A\xto{\gamma_A} \hat A\to\bsP A$. The doctrine axioms can be checked in a similar fashion as the previous cases.
\end{eg}
\begin{rmk}[All the previous examples, $\cV$-enriched]
	Unsurprisingly, most of these examples carry over for $\cV$-enriched categories, as discussed in  \cite[Chap. 5]{Kelly2005}. The $\cV$-enriched presheaf construction is analysed in \cite{day2007limits}; the enriched analogue of Ind-completion is quite old, but is elucidated in \cite{lack2021flat}; \cite{Kelly2005} contains a discussion of the enriched analogue of \autoref{finitecolimits}; the completion under absolute colimits in the enriched setting is also in \cite{Kelly2005}.
\end{rmk}
\begin{eg}[The Yoneda-Cantor embedding]
	Let us now consider the $2$-category $\sf Pos$ of posets, monotone mappings, and the partial order on every hom-set ${\sf Pos}(P,Q)$ defining the 2-cells; this works as a thinner version of the $2$-category of locally small categories and poses no size-issue whatsoever. The down-graded presheaf construction can be simplified in a very instructive manner. Let $\mathbb{T} = \{0 < 1\}$ be the usual two-element poset, and
	\[\mathbb{T}^{(\firstblank)^\op} : P \mapsto \Pos(P^\op,\mathbb{T})\]
	the set of antitone functions $P\to T$.

	If we regard $\Pos$ as the category of categories enriched over $\mathbb{T}$; and the unit of the doctrine is given by the \textit{weighted Cantor embedding} $C : P \to \mathbb{T}^{P^\op}$ which maps an element $p\in P$ to the characteristic function $\chi_{\le p}$ of the lower set $\{x\mid x\le p\}$, then the pair $(\mathbb{T}^{(\firstblank)^\op}, C)$ is a doctrine in the sense of \autoref{defkzdoctrine}.

	This construction provides the free completion under suprema in $\Pos$, regarded as the category of $\mathbb T$-categories, and it verifies the doctrine axioms.
\end{eg}

We now turn to the notion of \emph{cocomplete object} for a doctrine, again drawing from \cite{walker}.

\begin{defn}[Cocomplete object]\cite[Def. 4]{walker}\label{cocomplete_obj}
	Given a doctrine $(\bsS,\sigma)$ on a 2-category
	$\cK$, we say an object $ X\in\cK$ is \emph{$\bsS$-cocomplete}
	if for every $g\colon B\to X$
	\[
		\vcenter{\xymatrix{
				\bsS B\ar[r]^{\bar g} & X\dltwocell<\omit>{<2>c_g} & \bsS A \ar[r]^{\bar f} & \bsS B\dltwocell<\omit>{<2>c_f}\ar[r]^{\bar g} & X \\
				B \ar[u]^{\sigma_B}\ar[ur]_{g} && A\ar[u]^{\sigma_A} \ar[ur]_f
			}}
	\]
	there exists a left extension $\bar{g}$ as on the left diagram exhibited
	by an isomorphism $c_{g}$, and moreover the extended map $\bar g$ preserves the left extension $\bar f : \bsS A \to \bsS B$ of any 1-cell $f : A \to \bsS B$ along $\sigma_A$, as in the right diagram.
\end{defn}

\begin{defn}[$\bsS$\hyp{}cocontinuous $1$-cell]\label{cocont_ar}\cite[Def. 4]{walker}
	We say a 1-cell $E : X \to Y$ between $\bsS$-cocomplete objects $X$ and $Y$ is \emph{$\bsS$\hyp{}cocontinuous}  when it preserves all left extensions along $\sigma_B$ into $X$ for every object $B$.
\end{defn}

\begin{rmk}[Cocomplete objects and pseudoalgebras]\label{cocompl_psalg}
	As our running examples are all \textit{free cocompletions} under some family of colimits, it is not hard to guess the correct intuition behind the notion of a cocomplete object for a doctrine. We can summarise the concepts encountered so far in the following table:

	\begin{table}[!h]
		\begin{center}
			\begin{tabular}{ll}
				\textbf{Doctrine}      & \textbf{Cocomplete objects}                 \\ \toprule
				$\bsP$                 & cocomplete categories                       \\ \midrule
				$\bsInd_\lambda$       & categories with $\lambda$-directed colimits \\ \midrule
				$\bsD_\lambda$         & $\lambda$-cocomplete categories             \\ \midrule
				$\widehat{(-)}$ & categories with absolute colimits                         \\ \midrule
				$\mathbb{T}^{(-)^\op}$ & posets with suprema                         \\ \bottomrule
			\end{tabular}
		\end{center}
	\end{table}
	Let us stress that cocomplete objects for a given \kz $\bsS$ are precisely the \emph{pseudoalgebras} of the pseudomonad associated to $\bsS$ \cite[Prop. 6]{walker}. Analogously, $\bsS$\hyp{}cocontinuous $1$-cells are precisely pseudomorphisms of pseudoalgebras.
\end{rmk}
\begin{rmk}[Cocompleteness of free objects]
	With the previous remark in mind it comes without surprise that a free algebra, i.e. an object of the form $\bsS B$ for some $B$, is $\bsS$-cocomplete. This follows on the spot from \ref{kz_3} and \ref{kp_2}.
\end{rmk}
\subsection{The perks of being \texorpdfstring{$\bsS$}{S}-admissible}
\begin{rmk}
	Walker's paper draws a clear connection between doctrines and \emph{Yoneda structures} \cite{street1978yoneda}. Walker is informed about the work of Bunge and Funk \cite{bunge}, where the authors introduce the notion of \emph{admissible $1$-cell} for a doctrine, mimicking the homonymous notion that Street and Walters give for a Yoneda structure. Morally, admissibility is a smallness request on a 1-cell $f :A \to B$; admissibility for a doctrine will play a fundamental r\^ole in all that follows.
\end{rmk}
\begin{defn}[$\bsS$-Admissible $1$-cell]
	\label{defadm} Given a \kz $(\bsS,\sigma)$ on a 2-category $\cK$, we say a 1-cell $f\colon A\to B$
	is \emph{$\bsS$\hyp{}admissible} when  there exists a left extension $(N_f,\varphi_f)$ of
	$\sigma_A$ along $f$ as in the left diagram,
	\[\label{dis}
		\vcenter{\xymatrix{
		B \ar[r]^{N_f} & \bsS A\dltwocell<\omit>{\varphi_f} & B \ar[r]^{N_f} & \bsS A\ar[r]^{\bar h} \dltwocell<\omit>{\varphi_f} & X \dltwocell<\omit>{<2>c_h} \\
		& A\ar[u]_{\sigma_A}\ar@/^1.5pc/[ul]^f && A\ar[u]^{\sigma_A}\ar[ur]_h \ar@/^1.5pc/[ul]^f\\
		}}
	\]
	and moreover
	the left extension is preserved by any $\bar{h}$, as in the right
	diagram, where $X$ is an $\bsS$-cocomplete object.
\end{defn}
\begin{notat}
	For historical reasons, rooted in algebraic topology, we will call $N_f$ the \emph{nerve} of $f$.
	Evidently, $N_f =\Lan_f\sigma_A$ depends on $A$. When we need to make explicit which doctrine we are considering we use the notation $N_{(\bsS,f)}$.
\end{notat}
\begin{rmk}[Kancellation formula]\label{canc_rule}
	When dealing with equations between Kan extensions, admissible arrows are particularly well behaved. These are precisely those arrows for which the \emph{Kancellation formula} holds:
	\[\lan_f g \cong \lan_{\cancel{\sigma_A}} g \circ \lan_f \cancel{\sigma_A}\]
\end{rmk}
\begin{proof}
	Let $f: A \to B$ be an $\bsS$\hyp{}admissible arrow as in \autoref{defadm}, and $g: A \to X$ be any arrow into a $\bsS$-cocomplete object. We can fill the diagram
	\[
		\vcenter{\xymatrix@R=1.1cm@C=1.1cm{
		A \ar@{}[dr]|(.3){\Nearrow}\ar[r]^g \ar[d]_f & X \\
		B\ar[r]_{N_f}\ar[ur]_{\lan_fg} & \bsS A\ar[u]
		}}
	\]
	The right-hand diagram of \eqref{dis} displays $\bar h \cdot N_f$ as the left extension of $h$ along $f$, because the left extension $\varphi_f$ is preserved by $\bar h$ and $c_h$ is an isomorphism by assumption. So, since $\bar h = \lan_{\sigma_A} h$ and $N_f = \lan_f \sigma_A$, we get that
	\[\lan_f g \cong \lan_{\sigma_A} g \circ \lan_f \sigma_A.\qedhere\]
\end{proof}
\begin{notat}
	We shall denote $\mathsf{Adm}(\bsS)$ the subcategory of admissible arrows.\footnote{It is a class closed under composition, see \cite{walker}.} Given an object $A$, we shall denote by $\mathsf{Adm}_A(\bsS)$ the family of admissible arrows whose domain is $A$.
\end{notat}
A thorough study of admissibility in the 2-category of small categories for the small pre\-she\-a\-ves construction $\bsP$ of \autoref{smallpre} seems to be currently absent from the literature; we collect the main result on the topic in a short appendix at the end of this work (cf. \autoref{char_adm}, \autoref{solsetcond}).

The reader that is not used to the notion of admissibility might find the notion unnatural. Therefore, we dedicate our appendix to providing some properties of $\bsP$\hyp{}admissible functors. Here let us stress that a vast family of functors is admissible for the doctrine of small presheaves. For example, functors between small categories are all $\bsP$-admissible, functors with (small) arity are admissible, and thus accessible functors between accessible categories are all $\bsP$-admissible.
\begin{rmk}[Doctrines and free completions under colimits]
	Doctrines arising as free completions under some class of colimits are `generic' in the following sense. Let $\bsS$ be a doctrine on $\CAT$, whose units $\sigma$ are $\bsP$\hyp{}admissible (if $\bsP$ is the small presheaf construction); then the `nerve' $\lan_\sigma \yon $ is fully faithful,\footnote{As a consequence of the fact that $\sigma$ is `dense', meaning that the left extension of $\sigma$ along itself is the identity 1-cell of $\bsS A$.} and thus $\bsS A$ can be identified with a full subcategory of $\bsP A$, the free completion under small colimits \cite[Thm 16]{power2000representation}. Since being $\bsP$\hyp{}admissible can be seen as a very mild assumption, this observation shows that a vast majority of doctrines can be seen as a free completion under some family of colimits.

\end{rmk}

\begin{defn}[Formal inverse images functors]
	Given an $\bsS$\hyp{}admissible 1-cell $f\colon A\to B$
	we denote by $\bsS^*f$ the nerve of the composite 1-cell $\sigma_Bf$, which exists by the cocompleteness of $\bsS A$.
	\[
		\vcenter{\xymatrix{
				A\ar@{}[dr]|(.3){\Nearrow}\ar[d]_{\sigma_A}\ar[r]^f & \ar[dl]^{N_f} B \ar[d]^{\sigma_B}\\
				\bsS A & \bsS B \ar[l]^{\bsS^*f}
			}}
	\]
\end{defn}
\begin{rmk}
	The previous definition is implicitly observing that if $f$ is admissible, also $\sigma_B \circ f$ is so, because its admissibility reduces to the admissibility of $f$ and the cocompleteness of $\bsS A$.
\end{rmk}
\begin{rmk}[Inverse images are $\bsS$\hyp{}cocontinuous]\label{invimgiscocont}
	When it exists (that is, when $f$ is $\bsS$\hyp{}admissible) $\bsS^* f$ is, by definition, an  $\bsS$\hyp{}cocontinuous 1-cell.
\end{rmk}

\begin{rmk}[The nerve-realisation paradigm] \label{nerve-realisation-para}
	Notice that a $\bsS$\hyp{}admissible arrow $f: A \to B$ induces an adjunction \[\bsS_! f : \bsS A \leftrightarrows \bsS B : \bsS^* f.\]
	This follows directly from our definitions and \cite[Lemma 12]{walker}.
\end{rmk}
\begin{rmk}[Again on admissible $1$-cells for $\bsP$ on $\Cat$]
	We can use this last definition to have a better intuition on $\bsP$\hyp{}admissibles on  $\Cat$. Indeed by the Yoneda Lemma and the characterisation in \autoref{char_adm} it is easy to show that \[\bsP^* f \cong (-) \circ f.\] Thus, a $1$-cell is admissible if and only if for every small presheaf $P \in \bsP B$ the composition $P \circ f$ is still small.
\end{rmk}

\subsection{KZ {\yc}s}

In this subsection, we shall consider the following structure: a pair of \kzs $\bsP$ and $\bsS$ on a 2-category $\cK$, whose objects we think of as locally small categories of some kind, having respective unit components
\[
	\yon_A : A\to \bsP A \qquad \sigma_A : A \to \bsS A.
\]
As already said, the arrow $\yon_A : A\to \bsP A$ is intended as an abstraction of the Yoneda embedding of \autoref{smallpre}. On the other hand, we think $\bsS A$ as an abstraction of the $\bsInd$-completion of \autoref{bsind}; see \autoref{paradigmatico}.
\begin{defn}[KZ Context]\label{def_context}
	Let $\cK$ be a $2$-category. A \emph{\yc} on $\cK$ is a triple $(\bsP,\bsS,\upsilon)$ where:
	\begin{enumtag}{cx}
		\item \label{cx_1} the functors $\bsP$ and $\bsS$ are doctrines on $\cK$ whose unit components $\yon_A : A \to \bsP A$ and $\sigma_A : A \to \bsS A$ are fully faithful in the sense of \autoref{fufai}.
		\item \label{cx_2} For all objects $A$ the object $\bsP A$ is $\bsS$-cocomplete in the sense of \autoref{cocompl_psalg} and for all arrows $f: A \to B$ the arrow $\bsP_!(f) = \lan_{\yon_A}(\yon_B \circ f)$ is $\bsS$\hyp{}cocontinuous in the sense of \autoref{cocont_ar}.
		\item \label{cx_3} The arrow $\upsilon_A : \lan_{\sigma_A}(\yon_A) : \bsS A \to \bsP A$, whose existence is ensured by the assumption that $\bsP A$ is $\bsS$-cocomplete, is representably fully faithful in the sense of \autoref{fufai}.
	\end{enumtag}
\end{defn}
When there is no ambiguity of sort, we refer to a \yc just with the letter $\upsilon$: this uniquely determines the pair $(\bsP,\bsS)$ as codomain-domain of $\upsilon$.
\begin{rmk}
	Notice that from \ref{cx_3} it follows that $\upsilon_A$ is the $A$-component of a transformation $\upsilon : \bsS \To \bsP$, which is (pseudo)natural by the universal property of left extensions and \ref{cx_2}.
\end{rmk}
\begin{eg}\label{paradigmatico}
	The paradigmatic example of such a situation will be for us the couple $(\bsInd_\lambda, \bsP)$ on $\CAT$, where $\bsInd_\lambda$ is the doctrine of \autoref{bsind} and $\bsP$ is the small presheaves construction of \autoref{smallpre}.

	Concretely, using the identification between the Ind-completion $\bsInd_\lambda(A)$ of $A$ and the category $\Cat_{\lambda,\flat}(A^\op,\Set)$ of $\lambda$-flat small functors $A^\op\to\Set$, we obtain a description of $\upsilon_A$ given by the inclusion
	\[
		\upsilon_A: \Cat_{\lambda,\flat}(A^\op,\Set) \hookrightarrow \bsP A.
	\]
\end{eg}

\begin{eg}
	Similarly to the previous example, also the couple $(\bsD_\lambda, \bsP)$ where $\bsD_\lambda$ was defined in \autoref{finitecolimits}, is a context.
\end{eg}

\begin{rmk}[In a context the units of $\bsS$ are $\bsP$\hyp{}admissible]
	Axiom \ref{cx_2} of the definition of context $\sigma$ implies on the spot that the unit $\sigma_A: A \to \bsS A$ is $\bsP$\hyp{}admissible.
\end{rmk}
\begin{rmk}[A base change for admissible arrows]
	Given a context $(\bsS, \bsP)$, all the $\bsS$\hyp{}admissible arrows are $\bsP$\hyp{}admissible, i.e. $\mathsf{Adm}(\bsS) \hookrightarrow \mathsf{Adm}(\bsP)$. \autoref{lotsofadmissibles} will characterise the image of this inclusion, under the assumption that $\upsilon_A$ has a left adjoint.
\end{rmk}
The existence of the embedding above deserves a proof.
	\begin{proof}
	In the following diagram the composition $\upsilon \circ N_{(\bsS,f)}$ provides the $\bsP$-nerve of the 1-cell $f$.
	\[
		\vcenter{\xymatrix@R=1.1cm@C=1.1cm{
		A\ar@{}[dr]|(.3){\Nearrow} \ar[r]^f \ar[d]_{\sigma_A} & B\ar[dl]^{N_{(\bsS,f)}} \ar@{.>}[d]^{N_{(\bsP,f)}}\\
		\bsS A \ar[r]_{\upsilon_A} & \bsP A
		}}
	\]
	When we unpack the definitions, the last observation can be reduced to the `Kancellation formula' of \autoref{canc_rule}.
	\[N_{(\bsP, f)} = \lan_f \yon_A \stackrel{\ref{canc_rule}}{\cong} \lan_{\sigma_A} \yon_A \circ \lan_f \sigma_A = \upsilon_A \circ N_{(\bsS, f)} .\qedhere \]
\end{proof}
\begin{rmk}
	As a result of the previous observation, the class of $\bsS$ admissible arrows is, a priori, smaller then the class of $\bsP$\hyp{}admissible arrows, and we would be very much interested in understanding which $\bsP$\hyp{}admissible arrows happen to be $\bsS$\hyp{}admissible. We will come back to this question in the next section, where the existence of what we call \emph{envelopes} will help us to find some of such arrows.
\end{rmk}
\begin{eg}
	In $\CAT$, virtually every arrow is $\bsP$\hyp{}admissible, while much fewer arrows are $\bsInd$\hyp{}admissible.
\end{eg}
At least locally, we can characterise those objects where admissible arrows do not change under base change of doctrine. This proposition will be useful later.

\begin{prop} \label{lotsofadmissibles}
	Let $A$ be an object and assume that the component $\upsilon_A : \bsS A \to \bsP A$ has a left adjoint $l : \bsP A \to \bsS A$. Then the natural inclusion
	\[\xymatrix{\mathsf{Adm}_A(\bsS) \ar[r] & \mathsf{Adm}^{\upsilon l}_A(\bsP)}\]
	yields a bijection between $\bsS$\hyp{}admissible arrows and $\bsP$\hyp{}admissible arrows  whose $\bsP$-nerve is fixed by the monad $\upsilon_Al$.
\end{prop}
\begin{proof}
It is entirely clear that the inclusion lands in those $\bsP$\hyp{}admissible arrows whose $\bsP$-nerve is fixed by the monad $\upsilon_Al$, because for a $\bsS$\hyp{}admissible arrow, its $\bsP$-nerve coincides with $\upsilon N_{(\bsS,f)}$ and $l\upsilon \cong 1$. Now, when $\upsilon_A$ has a left adjoint $l_A$, by abstact nonsense, this must coincide with $N_{(\bsS, \yon_A)}$.  With this observation in mind, we can construct a candidate nerve by composition with $l_A$.
		      \[
			      \vcenter{\xymatrix@R=1.1cm@C=1.1cm{
			      A\ar@{}[dr]|(.3){\Nearrow} \ar[r]^f \ar[d]_{\sigma_A} & B\ar@{.>}[dl]^{N_{(\bsS,f)}} \ar[d]^{N_{(\bsP,f)}}\\
			      \bsS A  & \ar[l]^{l_A} \bsP A
			      }}
		      \]
The assumption that the $\bsP$-nerve is fixed by the monad $\upsilon_Al$ is precisely what ensures this candidate nerve to verify the second part of the \autoref{defadm}.
\end{proof}

We now introduce a notion of smallness for an object based on admissibility: the property of being \emph{petit} (evidently, French for `small') with respect to a dctrine $\bsS$. After the definition, we will describe $\bsP$-petit objects in $\CAT$.

\begin{defn}[Petit objet]\label{petit_obj}
	Let $(\bsS,\sigma)$ be a doctrine. We call an object $A$ of $\cK$ $\bsS$-\emph{petit} if every arrow $f :A \to X$ is $\bsS$\hyp{}admissible in the sense of \autoref{defadm}.
\end{defn}
\begin{eg}
	It is well known that small categories are $\bsP$-petit in the context described in \autoref{smallpre}.
\end{eg}

\subsection{Accessibily and Presentability}\label{sec_acc_pres}
The fundamental idea behind the definition of a locally presentable and accessible category is abstracting the property of being determined by a small set of information (the property of being `bounded' by a regular cardinal $\lambda$) into a clean categorical request.

We noe reached the heart of the present paper, where we provide a convincing definition of $\lambda$-accessible and locally $\lambda$-presentable in a general $2$-category equipped with a context $(\bsS, \bsP$). There are several equivalent characterisations of both accessibility and presentability in the literature; thus, we should discuss which one of these characterisations we are abstracting.
\begin{itemize}
	\item A category $A$ is $\lambda$-accessible if and only if it is equivalent to a category $\bsInd_\lambda(G)$ for some small $G$ (cf. \cite[2.26]{Adamek1994}).
	\item A category $K$ is (locally) $\lambda$-presentable if it is a $\lambda$-accessible reflective full subcategory of a presheaf category over a small category $G$, i.e. if there is a reflection of the form $K \leftrightarrows \bsP G$ where the right adjoint $K\to \bsP G$ is a $\lambda$-accessible functor (cf. \cite[1.46]{Adamek1994}).
\end{itemize}
This paves the way to the following two definitions.
\begin{defn}[$\upsilon$-accessible object]\label{yonacc}
	Given a \yc $\upsilon$ on the 2-category $\cK$, an object $A\in\cK$ is $\upsilon$-\emph{accessible} if there exists a $\bsP$-petit object $G\in \cK$ such that $A \simeq \bsS G$.
\end{defn}
\begin{defn}[$\upsilon$-presentable object]\label{yonpres}
	Let $\upsilon$ be a \yc; an object $A\in\cK$ is $\upsilon$-\emph{presentable} if it is an $\upsilon$-accessible reflective full subobject of some $\bsP G$, for some $\bsP$-petit object $G$, and such that the inclusion $A \hk \bsP G$ is $\bsS$\hyp{}cocontinuous.
\end{defn}
It is evident how the nomenclature introduced so far has been engineered to state analogues of these results in an abstract 2-category endowed with a context. Only one small detail is left open.
\begin{rmk}[Why `locally', and why not?]
	The historically correct name for presentable categories contains the adjective `locally' because those categories are \emph{locally} specified by \emph{presentable} objects. Although, more recently, some authors prefer to drop the adjective \emph{locally}, hinting at the idea that a presentable category is a cocomplete category specified by a small amount of data. We choose this second option: thus, for us a `presentable category' is a $\upsilon$-presentable 0-cell of the 2-category $\Cat$, with respect to the context $\upsilon : \bsInd_\lambda \To \bsP$.
\end{rmk}

For the sake of clarity, let's unwind our definition \autoref{yonpres} above; an object $A$ is $\upsilon$-presentable if the following conditions are met:
\begin{enumtag}{p}
	\item \label{p:uno} there exists an adjunction $L : \bsP G \leftrightarrows A : i$, where $i$ is representably fully faithful and $G$ is a $\bsP$-petit object;
	\item \label{p:due} $A$ is a $\upsilon$-accessible object in the sense of \autoref{yonacc};
	\item \label{p:tre} $i$ is an $\bsS$\hyp{}cocontinuous functor in the sense of \autoref{cocont_ar}.
\end{enumtag}
In the following, we freely refer as properties \ref{p:uno}--\ref{p:tre} as part of the definition of $\upsilon$-presentability.
\subsection{Representation theorem}
Now that we have given the definition of accessibilty and presentability, we should provide evidence that this notion has something to do with the classical one. We choose to reproduce one of the most representative results in the theory of locally presentable categories, which we state below.
\begin{thm}[Representation theorem]
	Let $A$ be a locally small categories. Then the following conditions are equivalent
	\begin{enumtag}{cp}
		\item $A$ is cocomplete and $\lambda$-accessible;
		\item $A$ is $\lambda$-presentable.
	\end{enumtag}
\end{thm}
Of course, we have higher ambitions than just reproducing this theorem; but a more faithful simulation of the classical theory will require additional assumptions: in the next section, we will show how to recover Gabriel-Ulmer duality in a 2-category endowed with a context.
\begin{thm}[Formal representation theorem]\label{the-main}
	Let $\upsilon$ be a \yc. Then the following conditions are equivalent for an object $A\in\cK$:
	\begin{enumtag}{rt}
		\item $A$ is $\bsP$-cocomplete and $\upsilon$-accessible;
		\item $A$ is $\upsilon$-presentable.
	\end{enumtag}
\end{thm}
\begin{proof}
	Assume that $A\in\cK$ is presentable; the fact that $A$ is accessible is the content of \ref{p:due}. To show that $A$ is cocomplete, we can observe that reflective subobjects of $\bsP$-cocomplete objects remain $\bsP$-cocomplete. This proves the first implication.

	Now assume that $A$ is $\upsilon$-accessible and $\bsP$-cocomplete; we shall show that the axioms \ref{p:uno}--\ref{p:tre} hold. By definition of $\upsilon$-accessibility, $A$ is of the form $\bsS(G)$, for some $G$ which is $\bsP$-petit, so we can draw the following diagram filled by an isomorphism.
	\[
		\vcenter{\xymatrix{
				& G  \ar[dr]^{\sigma_G}\ar[dl]_{\yon_G}& \\
				\bsP G &\utwocell<\omit>{}& \ar[ll]^{\upsilon_G} \bsS  G \simeq A
			}}
	\]
	Also, by definition of context, $\upsilon_G$ is $\bsS$\hyp{}cocontinuous and representably fully faithful. Since $A$ is $\bsP$-cocomplete, $\upsilon_G$ has a left adjoint, by \cite[Lemma 12]{walker}. This shows \ref{p:due} and \ref{p:tre}.
\end{proof}

\section{Gabriel Ulmer Duality}\label{sec:gabriel}
The present section in structured in four subsections.
\begin{itemize}
	\item In \autoref{guenvelope} we introduce the notion of \emph{Gabriel-Ulmer envelope}, together with the notion of context; this is a key ingredient of our theory and will help us in building a formal version of Gabriel-Ulmer duality.
	\item In \ref{morita_in_ctx} we discuss the notion of Morita equivalent objects for a doctrine. These are objects $A,B$ such that $\bsS A \simeq \bsS B$.
	\item We use this technology to provide a sharper characterisation of $\upsilon$-presentable objects, in \autoref{production}.
	\item We provide a context-relative version of the Gabriel-Ulmer duality in \autoref{guduality}.
\end{itemize}

\subsection{Envelopes}
Given a category $A$, denote with $\bsD A$ the subcategory of $[A^\op,\Set]$ spanned by finite colimits of representables. As mentioned before, in \autoref{finitecolimits}, this is the free completion of $A$ under finite colimits. Clearly $(\bsD A)^\op$ is the free completion of $A^\op$ under finite \emph{limits}. This means that \[\mathsf{Lex}((\bsD A)^\op,\Set)\simeq [A^\op,\Set],\] where $\mathsf{Lex}(A,B)$ is the usual notation for the category of functors preserving finite limits, borrowed for example from \cite{makkai1987some}.

When a category has finite colimits, then we have a simplified description of its $\bsInd$\hyp{}completion, which is given by the fact that flat functors are precisely functors preserving finite limits (cf. \cite[6.3.2]{Bor1}). In other words, \emph{there exists an object $\bsD A$ such that} \[\bsInd(\bsD A) \simeq \bsP A.\]

In trying to abstract the above equation to the setting of a generic 2-category $\cK$ equipped with \kzs $\bsP,\bsS,\bsD$, this amounts to a factorisation of $\yon_A : A \to \bsP A$ as a composition $A\to \bsD A \to \bsS (\bsD A)\simeq \bsP A$, naturally in $A\in\cK$. This will turn out to be a fundamental property, and motivates us to define what follows.
\begin{defn}[\gu envelope]\label{guenvelope}
	A (Gabriel\hyp{}Ulmer) \emph{envelope}, shortly a \emph{\gu envelope} or just an \emph{envelope}, relative to a context $\upsilon : \bsS\To \bsP$ consists of an additional \kz{} $\bsD$ with fully
	faithful unit $\delta_A : A \to \bsD A$ such that:
	\begin{enumtag}{g}
		\item \label{g:1} The pair $(\bsP, \bsD)$ is also a \yc in the sense of \autoref{def_context}.
		\item \label{g:2} For all objects $A\in\cK$, the object $\bsS(\bsD A)$ is $\bsP$-cocomplete and the arrow $\bsP A \to \bsS(\bsD A)$ induced from the composite $A \to \bsD A \to \bsS(\bsD A)$ by left extension along $\yon_A$ is an equivalence.
	\end{enumtag}
\end{defn}
Condition \ref{g:2} deserves t be spelled out more explicitly. It requires that the diagram
\[\label{theta_for_gu2}
	\vcenter{\xymatrix{
	&A\ar[rr]^{\yon_A} \drtwocell<\omit>{\theta}\ar[dl]_{\delta_A}&&\bsP A\ar@{.>}[dl]^{\iota_A}\\
	\bsD A\ar[rr]_{ \sigma_{\bsD A}} && \bsS \bsD A &
	}}
\]
is filled by the dotted 1-cell $\iota_A : \bsP A \to \bsS\bsD A$ and by an invertible 2-cell $\theta : \iota_A \cdot\yon_A \cong \sigma_{\bsD A}\circ\delta_A$: the assumption that $\bsS\bsD A$ is $\bsP$-cocomplete ensures the existence of the left extension $\iota_A$ of $\sigma_{\bsD A}\circ\delta_A$ along $\yon_A$. The 2-cell filling the diagram must be invertible, because $\yon_A$ is fully faithful (cf. \ref{cx_1} of \autoref{def_context}), and \ref{g:2} requires that the 1-cell $\iota_A$ is invertible.
\begin{eg}
	The guiding example in giving this definition is the 2-category $\CAT$, which has a \gu envelope, relative to the standard context $\bsInd_\omega\to [(\firstblank)^\op,\Set]$, defined sending $A$ into its finite colimit completion $\bsD A$.
\end{eg}
\subsection{Morita theory in context}\label{morita_in_ctx}
To proceed with our analysis, we need to introduce the notion of \emph{Cauchy completeness} for an object $A\in\cK$: classical references about the Cauchy completion of categories are \cite{LawvereFW:metsgl,CTGDC_1983__24_4_377_0,CTGDC_1986__27_2_133_0}; it will appear evident to the reader that our intent here is to provide a formal category-theoretic account of Cauchy completeness.
\begin{defn}\label{caucau}
	Let $\bsS$ be a \kz on $\cK$. We say that $A$ is $\bsS$-\emph{Cauchy complete} if and only if any adjunction $q^* :  \bsS B \leftrightarrows \bsS A : q_* $ in which $q_*$ is $\bsS$\hyp{}cocontinuous \emph{converges}, in the sense that there exists some, necessarily $\bsS$-admissible, $f : B\to A$ such that $q^* =\bsS_! f$.
	\[
		\vcenter{\xymatrix@R=1.1cm@C=1.1cm{
		A\ar[d]_{\sigma_A} & \ar@{.>}[l]_f B \ar[d]^{\sigma_B}\\
		\bsS A \ar@<-.5em>[r]_{q_*}\ar@{}[r]|\perp & \ar@<-.5em>[l]_{q^* \cong \bsS_! f}\bsS B
		}}
	\]
\end{defn}

\begin{rmk}[$\bsP$-Cauchy complete objects in $\Cat$]
	The terminology is motivated by the fact that in $\cK=\Cat$, a category is Cauchy complete (in the sense that all idempotents split, or equivalently that $A$ coincides with the subcategory of tiny objects \cite{CTGDC_1986__27_2_133_0} in $\bsP A$) if and only if it is $\bsP$-Cauchy complete in the sense of \autoref{caucau}, where $\bsP$ is the small presheaves construction.
\end{rmk}
\begin{proof}
	If $A$ is Cauchy complete, then an adjunction $p\dashv q : B \xto{p} A$, where $q$ is $\bsP$\hyp{}cocontinuous is such that $p$ preserves tiny objects, thus inducing a restricted functor $p|_B$ in
	\[
		\vcenter{\xymatrix{
		B \ar[d]\ar[dr]^f\ar@/_2pc/[dd]_{\yon_B}\ar@{.>}[r]^{p|_B}& \bar A\ar[d]^\wr \\
		\bar B \ar[d] & A\ar[d]^{\yon_A}\ar[dl]_{N_f}\\
		\bsP B \ar[r]_p & \bsP A
		}}
	\]
	that must induce the adjunction $p\dashv q$ as $\lan_{\yon_B}(\yon_A f)\dashv \lan_{\yon_A f}\yon_B$; $q$ is cocontinuous because it has the right adjoint $\lan_{N_f}\yon_A$, where $N_f=\lan_f\yon_B$.

	On the other hand, if $A$ is $\bsP$-Cauchy complete in the sense of \autoref{caucau}, consider the embedding $u : A\to \bar A$ of $A$ in its Cauchy completion; this must induce an equivalence $\bsP A\simeq \bsP \bar A$, so in particular, there must exist a $\bar A \to A$ left adjoint to $u$; thus, $A$ is Cauchy complete.

	This last point deserves more explanation: Cauchy completion can be seen as absolute cocompletion (i.e. completion under absolute colimits, cf. \cite{Kelly2005}). Free cocompletions are KZ monads, as shown by Kock in \cite{kock1995monads}. Since the property of being cocomplete can be rephrased as `being a pseudo algebra for the associated KZ monad', and such structure only depends on the existence of a left adjoint for the unit, we conclude that $A$ is Cauchy complete.
\end{proof}
\begin{rmk} One way to synthesize the proof above in a way that is more digestible to the reader is the following:
	\begin{quote}
		Given any adjunction $q^* \dashv q_*$ where the right adjoint preserves all colimits, the left adjoint must send tiny objects to tiny objects.
	\end{quote}
Thus, if tiny objects coincide with the image of the Yoneda embedding, one has that $A$ is $\bsP$-Cauchy complete.
Yet, this condition is identical to the request that $A$ is Cauchy complete in the classical sense.
\end{rmk}

\begin{rmk}[$\bsInd_\lambda$-Cauchy complete in $\Cat$] \label{ccnoncambia} Given an adjunction $q^* \dashv q_*$ in $\Cat$ where the right adjoint preserve $\lambda$-directed colimits, the left adjoint will always map $\lambda$-presentable objects to $\lambda$-presentable objects. Thus, following the previous discussion, a category $A$ will be $\bsInd_\lambda$-Cauchy complete if the $\lambda$-presentable objects in $\bsInd_\lambda(A)$ coincide with the image of the Yoneda embedding. It is easy to show that an object is $\lambda$-presentable in $\bsInd_\lambda(A)$ if and only if it is a retract of a representable object, the proof is identical to the case of the presheaf construction and thus $\bsInd_\lambda$-Cauchy complete are again Cauchy complete categories in the classical sense.
\end{rmk}
\begin{defn}
	We define the following locally full sub-2-categories of:
	\begin{enumtag}{k}
		\item\label{k_1} $\cK_{\bsS,!}$, the $\bsS$\hyp{}cocontinuous arrows (cf. \autoref{defadm}),
		\item\label{k_2} $\cK_{\bsS,\mathrm{a}}$, the $\bsS$-admissible  (cf. \autoref{cocont_ar}), and
		\item\label{k_3} $\cK_{\bsS,\ell}$, those arrows that have $\bsS$\hyp{}cocontinuous right adjoints.
	\end{enumtag}
	Of course, $\cK_{\bsS,\ell} \subseteq \cK_{\bsS,!}$. We allow intersection of sorts whenever necessary, for example $f\in\cK_{\bsS,\mathrm{a!}}$ is a cocontinuous, admissible arrow.
\end{defn}
Then the following result follows directly from the fact that $\bsS$ is a doctrine with fully faithful unit:
\begin{lem}
	The (corestricted) 2-functor $\bsS_! : \cK \to \cK_{\bsS,!}$ is locally fully faithful, so an object $A$ is $\bsS$-Cauchy complete if and only if the functor
	\[\bsS_! : \cK_{\bsS,\mathrm{a}}(B,A) \to \cK_{\bsS,\ell}(\bsS B,\bsS A)\]
	is an equivalence for all objects $B \in \cK$.
\end{lem}
\begin{proof}
	By \autoref{nerve-realisation-para} and \autoref{invimgiscocont}, the corestricted functor $\bsS_!$ is well defined (one can routinely check that the action on $2$-cells is induced by the universal property of the left extension and that it is contravariant). Consider now a diagram as in \autoref{caucau}, since $\sigma_A$ and $\sigma_B$ are fully faithful, the functor is faithful, and since they are dense the functor is full. \autoref{caucau} is saying precisely that the functor in question is also essentially surjective; this concludes the proof.
\end{proof}
The notion of Cauchy-completeness yields naturally to an abstraction of the relation of Morita equivalence in a context. If $\upsilon : \bsS \To \bsP$ is such a context, we can prove the following results.
\begin{cor}[$\bsP$-Morita equivalence] \label{moritaeq}
	If $A$ and $B$ are $\bsP$-Cauchy complete then for any equivalence $u : \bsS A \simeq \bsS B$ there exists an equivalence $v: A \simeq B$ such that $\bsS_! v \cong u$.
\end{cor}
\begin{proof}
 It follows at once by the previous Lemma.
\end{proof}
\begin{cor} \label{moritaleftadj}
	Suppose that $A$ is $\bsS$-Cauchy complete and that $u: B \to A$ is an arrow with the property that $\bsS_! u: \bsS A \to \bsS B$ admits a left adjoint then $u$ itself admits a left adjoint.
\end{cor}
\begin{proof}
	The arrow $\bsS_! : \bsS B \to \bsS A$ is $\bsS$\hyp{}cocontinuous, since it is induced by the universal property of $\bsP B$. By assumption it has a left adjoint $l : \bsS A \to \bsS B$, so we may apply the $\bsP$-Cauchy completeness of $B$ we obtain an arrow $v : A \to B$ with $l \cong \bsS_! v$ and hence $\bsS_! v\dashv \bsS_! u$. The 2-functor $\bsS_! : \cK_{\bsS} \rightarrow \cK_{\bsS,!}$ is locally fully faithful so the unit and counit of the adjunction $\bsS_! v \dashv \bsS_! u$ lies in its image and thus may be lifted along it to give unit and counit for an adjunction $v\dashv u$ in $\cK_{\bsS}$.
\end{proof}
\subsection{On the consequences of having an envelope}
	The next remark is another step in our simulation process of the classical theory of (locally) $\lambda$-presentable categories. The result is quite well-known and follows from a combination of theorems in \cite{Adamek1994}.
	\begin{rmk}
		The following are equivalent for a category $G$:
	\begin{enumtag}{l}
		\item\label{lam_1} $G$ has finite colimits;
		\item\label{lam_2} $\bsInd \, G$ is locally finitely presentable.
		\item\label{lam_3} $\bsInd \, G$ is cocomplete.
	\end{enumtag}
	When a context $\upsilon$ admits an envelope $(\bsD,\delta)$, we manage to recover this result.
\end{rmk}
\begin{prop} \label{production}
	Let $G\in\cK$ be an $\bsS$-Cauchy complete object, and $\upsilon : \bsS \To \bsP$ a \yc  for which there exists an envelope $\bsD$ as in \autoref{guenvelope}; then the following conditions are equivalent.
	\begin{enumtag}{pc}
		\item \label{pc:uno} $G$ is $\bsD$-cocomplete (cf. \autoref{cocomplete_obj});
		\item \label{pc:due} $\bsS G$ is $\bsP$-cocomplete;
		\item \label{pc:tre} $\bsS G$ is $\upsilon$-presentable (cf. \autoref{yonpres}).
	\end{enumtag}
\end{prop}
\begin{proof}
	The fact that \ref{pc:due} is equivalent to \ref{pc:tre} is evident thanks to \autoref{the-main}; to show that \ref{pc:uno} implies \ref{pc:due}, note that $G$ is $\bsD$-cocomplete if and only if it is a reflective subobject of $\bsD G$; given this, we have an adjunction
	\[
		\xymatrix{G \ar@{}[r]|\perp \ar@{^{(}->}@<-4pt>[r]_{i_G}& \bsD G. \ar@<-4pt>[l]_{l_G}}
	\]
	If we apply $\bsS$ to this adjunction and use the (adjoint) equivalence $\iota_G$ of \eqref{theta_for_gu2} we get an adjunction $\bsS G \leftrightarrows \bsP G$ back, from which we derive that $\bsS G$ is a reflective subobject of $\bsP G$, and thus it is $\bsP$-cocomplete.

	To conclude, we prove the converse, i.e. that \ref{pc:due} implies \ref{pc:uno}: the proof of this last implication will occupy the next page.

	Suppose that $G$ is $\bsS$-Cauchy complete and that $\bsS G$ is $\bsP$-cocomplete: then $G$ is $\bsD$-cocomplete.

	Now, consider the pasting diagram obtained from the following square and triangle
	\[\vcenter{\xymatrix{
		& G \ar[r]^{\delta_G}\ar[d]_{\yon_G}\ar@/_1pc/[dl]_{\sigma_G}& \bsD G\ar[d]^{\sigma_{\bsD G}} \\
		\bsS G \ar[r]_{\upsilon_G} & \ultwocell<\omit>{c}\bsP G\ar[r]_{\iota_G} & \bsS\bsD G\ultwocell<\omit>{\theta_G}
		}}\]
	where both 2-cells are invertible($\theta$ is as in \autoref{theta_for_gu2}, so $\iota_G$ is invertible).

	It is easy to observe that this pasting exhibits a left extension, as it results from pasting two left extensions, and that the 2-cell $\theta_G\cdot(\iota_G * c_{\yon_G})$ remains invertible. On the other hand, the same triangle exhibits $\bsS_!\delta_G$, as in \autoref{lr_shriek} and \autoref{def_context}.\ref{cx_2}; thus, we obtain an invertible 2-cell $\bsS_!\delta_G\cong \iota_G\cdot \upsilon_G$, and from this we derive
	\[i\cdot\sigma_G = \iota_G^{-1} \cdot \bsS_!\delta_G\cdot\sigma_G\cong \upsilon_G\cdot\sigma_G\cong\yon_G.\]
	Furthermore, under the assumption that $\bsS G$ is $\bsP$-cocomplete the universal
	property of $\bsP G$ may be applied to induce a $\bsP$\hyp{}cocontinuous left extension $l : \bsP G \to \bsS G$ of $\sigma_G$ along $\yon_G$, displayed by an isomorphism $\beta : \sigma_G \cong l\circ\yon_G$ filling the diagram
	\[
		\vcenter{
		\xymatrix{
		G\ar[d]_{\yon_G}\ar[r]^{\sigma_G}\drtwocell<\omit>{\beta} & \bsS G \\
		\bsP G \ar@/_1pc/[ur]_l &
		}}
	\]
	Now, The identity $1_{\bsP G}$ is a left extension of $\yon_G$ along itself, so the composite isomorphism $(i * \beta)\circ\alpha : \yon_G \cong i\circ l\circ \yon_G$ induces a 2-cell $\eta : 1_{\bsP G} \To i\circ l$.

	Similarly, we get a 2-cell $\kappa : 1_{\bsS G} \To l\circ i$, because the identity $1_{\bsS G}$ is a left extension of $\sigma_G$ along itself. The situation is depicted in the following diagrams:
	\[{\footnotesize
		\vcenter{\xymatrix{
		G \rruppertwocell^{\sigma_G}{\beta}\ar[r]\ar[d]_{\sigma_G}\drtwocell<\omit>{\alpha}& \bsP G\ar[r]_l&\bsS G \\
		\bsS G\ar@/_1pc/[ur]_i&&
		}}=
		\vcenter{\xymatrix{
				G\ar[r]^{\sigma_G}\ar[d]_{\sigma_G} & \bsS G\\
				\bsS G\urlowertwocell_{l\cdot i}{\kappa}\ar@{=}[ur]
			}}
		\qquad
		\vcenter{\xymatrix{
		G\rruppertwocell^{\yon_G}{\alpha}\ar[r]\ar[d]_{\yon_G}\drtwocell<\omit>{\beta} & \bsS G \ar[r]_i& \bsP G\\
		\bsP G\ar@/_1pc/[ur]_l&&
		}}=
		\vcenter{\xymatrix{
				G\ar[r]^{\yon_G}\ar[d]_{\yon_G} & \bsP G\\
				\bsP G\urlowertwocell_{i\cdot l}{\eta}\ar@{=}[ur]
			}}}
	\]
	It is also easily seen that the context axioms imply
	that any $\bsP$\hyp{}cocontinuous arrow is $\bsS$\hyp{}cocontinuous; it follows that the arrow $l\circ i$
	is $\bsS$\hyp{}cocontinuous since it is a composite of the $\bsP$\hyp{}cocontinuous arrow l and the
	$\bsS$\hyp{}cocontinuous arrow $l$. So $\kappa$ is a 2-cell between $\bsS$\hyp{}cocontinuous arrows whose whiskering by $\sigma_G$ is an isomorphism, from which it follows that $\kappa$ itself is an isomorphism.

	A straightforward computation, starting with the defining equations for $\eta$ and $\kappa$ shows that $i\circ\kappa\circ \sigma_G = \eta\circ i \circ \sigma_G$.

	Furthermore, the domain $i$ and codomain $i\circ l \circ i\cong i$ of the 2-cells $i\circ \kappa$ and $\eta\circ i$ are both $\bsS$\hyp{}cocontinuous, so it follows from
	the universal property of $\sigma_G$, that they are equals. The dual argument, using the universal property of $\yon_G$, also shows that $\kappa \circ l = l\circ \eta$. Taking $\epsilon := \kappa^{-1}$, these two equations transform to give triangle identities demonstrating that $\eta$ and $\epsilon$ are unit and counit of an adjunction $l\dashv i$ with invertible counit.

	Finally, by canceling the equivalence $\bsS(\bsD G)\simeq \bsP G$ we see that $\lan_{\sigma_G}(\sigma_{\bsD G}\circ\delta_G)$ has a left adjoint $l'$ in the square
	\[
		\vcenter{\xymatrix@C=2cm{
		G\ar[r]^{\delta_G}\ar[d]_{\sigma_G} & \bsD G\ar[d]^{\sigma_{\bsD G}} \\
		\bsS G\ar@<-.5em>[r]_-{\lan_{\sigma_G}(\sigma_{\bsD G}\circ\delta_G)} \ar@{}[r]|\perp & \bsS(\bsD G)\ar@<-.5em>[l]_{l'}
		}}
	\]
	Finally, the assumption that $G$ is $\bsS$-Cauchy complete, we can obtain an adjunction $r \dashv \delta_G$ with invertible counit; this is equivalent to the fact that $G$ is $\bsD$-cocomplete.
\end{proof}
\begin{rmk}[On (right) adjoints to cells of the form $\bsS_! f$]\label{Zf-has-adjoints}
	For every $\bsD$\hyp{}cocontinuous, $\bsP$-admissible morphism $f : G\to G'$ between $\bsD$-cocomplete objects, there exists a 1-cell $\bsS^* f$ in diagram below which is right adjoint adjoint to $\bsS_! f$.
	\[\vcenter{\xymatrix{
		G \ar[r]^f\ar[d]_{\sigma_G} & G' \ar[d]^{\sigma_{G'}} \\
		\bsS G \ar@<.5em>[r]^{\bsS_! f}\ar@{}[r]|\perp & \ar@<.5em>[l]^{\bsS^* f}\bsS G'
		}}\]
	Moreover $\bsS^* f$ exhibits the universal property of $\lan_{ \sigma_{G'}f}( \sigma_G)$. Of course this follows directly from \autoref{lotsofadmissibles} and \autoref{production}, because in this case $\bsS G$ is cocomplete. Let us provide a more transparent description of such right adjoint.
	\[\vcenter{\xymatrix{
		G\ar[dr]^{ \sigma_G}\ar[rrr]^f \ar[dd]_{\yon_G} &&& G' \ar[dd]^{\yon_{G'}}\ar[dl]_{ \sigma_{G'}}\\
		& \bsS G\ar[dl]^{\upsilon_G} \ar@<4pt>[r]^{\bsS_! f}&\ar@{.>}@<4pt>[l]^{\bsS^* f} \bsS G' \ar[dr]_{\upsilon_{G'}} & \\
		\bsP G &&&\ar[lll]^{\bsP^* f} \bsP G'
		}}
	\]
	Call $L: \bsP G \to \bsS G$ the reflection of $\bsS G$, which exists because $\bsS G$ is $\upsilon$-presentable (this is implied by the fact that $G$ is $\bsD$-cocomplete and by \autoref{production}).  Then, writing explicitely the content of 2.31, we obtain that
	\[\label{the-claim}
		\bsS^* f \cong L \circ \bsP^* f \circ \upsilon_{G'}.
	\]
\end{rmk}
\subsection{Gabriel-Ulmer Duality in context}\label{gu_in_ctx}
Gabriel\hyp{}Ulmer duality as exposed in \cite{makkai1987some} builds a bi\hyp{}equivalence
\[
	\textsf{Mod} : \Lex^\op \leftrightarrows \LFP : \textsf{Th}
\]
between the 2-category $\Lex$ of small categories with finite limits, finite limit preserving functors, and natural transformations, and the 2-category $\LFP$ of locally finitely presentable categories, finitary right adjoint functors (i.e. functors $R : \cH\to\cK$ with a left adjoint, and preserving filtered colimits) and natural transformations.

The idea is that a finitely complete category $C\in\Lex$ is a ``theory'', whose category of models $\Lex(C,\Set)$ is locally finitely presentable. Gabriel\hyp{}Ulmer duality says that all locally finitely presentable categories arise in this way, as it is possible to extract an essentially unique theory of which a given $\cK\in\LFP$ is the category of models.

\begin{assume}\label{assume}
	Let us spell out the assumptions and the notation of this subsection.
	\begin{enumerate}
		\item $\upsilon : \bsS\To\bsP$ is a context on $\cK$ in the sense of \autoref{def_context};
		\item and $\bsD$ a \gu envelope in the sense of \autoref{guenvelope};
		\item for every $\upsilon$-accessible category $\bsS G$ there exists a $\bsS$-Cauchy complete and $\bsP$-petit object $\hat G$ and an arrow $i: G \to \hat G$ such that $\bsS i$ is an equivalence.
	\end{enumerate}
\end{assume}
\begin{rmk}
	In the leading example of this paper, that is the case of $\bsP$ and $\bsInd_\lambda$, the third condition is met by virtue of \autoref{ccnoncambia}. Indeed, the two Cauchy-completions coincide in this case and the Cauchy completion of a class category is still small.
\end{rmk}
\begin{defn}[The 2-category $\upsilon\text{-}\Rex$]
	We define the 2-category $\upsilon\text{-}\Rex$ having 0-cells the $\bsP$-petit, $\bsS$-Cauchy complete, $\bsD$-cocomplete objects, 1-cells the $\bsD$\hyp{}cocontinuous cells, and all 2-cells between them.
\end{defn}
\begin{defn}[The 2-category $\upsilon\text{-}\LP$]
	The objects of the 2-category $\upsilon\text{-}\LP$ are $\upsilon$-presentable objects of $\cK$: by our \autoref{the-main}, this class coincides with $\upsilon$-accessible and cocomplete 0-cells; 1-cells are right adjoints that are $\bsS$\hyp{}cocontinuous cells according to \autoref{cocont_ar}, with all 2-cells of $\cK$ between them.
\end{defn}
This leads to our main theorem:
\begin{thm}[Gabriel\hyp{}Ulmer duality]\label{guduality}
	Under the assumptions of the subsection (\autoref{assume}), there is a bi-adjunction
	\[
		\yomod : \upsilon\text{-}\Rex^\coop \rightleftarrows \upsilon\text{-}\LP : \yoth
	\]
	which is in fact a bi-equivalence of 2-categories.
\end{thm}
\begin{proof}
	We start defining the action of two functors $\yomod$ and $\yoth$; the first is defined ``applying $\bsS$'', meaning that its action on 0- and 1-cells is determined as follows:
	\[
		\vcenter{\xymatrix{
		G \ar@[lightgray]@{^{(}->}@<-4pt>[r]_{\gray{ i_G}} & \color{lightgray} \bsD G \ar@[lightgray]@<-4pt>[l]_{\gray{ l_G}} & \rightsquigarrow &
		\bsS G \ar@[lightgray]@{^{(}->}@<-4pt>[r]_{\gray{\bsS i_G}} & \color{lightgray} \bsS \bsD G \ar@[lightgray]@<-4pt>[l]_{\gray{\bsS l_G}}\\
		G \ar[d]_f \ar@[lightgray]@{^{(}->}@<-4pt>[r]_{\gray{ i_G}} & \color{lightgray} \bsD G \ar@[lightgray][d]^{\gray{ \widehat f}} \ar@[lightgray]@<-4pt>[l]_{\gray{ l_G}} & \rightsquigarrow & \bsS G'\ar@{<-}[d]_{\bsS^* f} \ar@[lightgray]@{^{(}->}@<-4pt>[r]_{\gray{\bsS i_{G'}}} & \color{lightgray} \bsS \bsD G' \ar@[lightgray]@<-4pt>[l]_{\gray{\bsS l_{G'}}}\ar@{<-}@[lightgray][d]^{\gray{ \bsP^* f}}\\
		G' \ar@[lightgray]@{^{(}->}@<-4pt>[r]_{\gray{ i_{G'}}} & \color{lightgray} \bsD G' \ar@[lightgray]@<-4pt>[l]_{\gray{ l_{G'}}}& \rightsquigarrow & \bsS G \ar@[lightgray]@{^{(}->}@<-4pt>[r]_{\gray{\bsS i_G}} & \color{lightgray} \bsS \bsD G \ar@[lightgray]@<-4pt>[l]_{\gray{\bsS l_G}}
		}}
	\]
	and on 2-cells it acts again as $\bsS$ ($\bsS^* f$ is the right adjoint to $\bsS_! f$ appearing in \autoref{Zf-has-adjoints}). We have to check that this really defines a functor taking values in $\upsilon\text{-}\LP$; this is easily seen, as every $\bsS G$ is $\bsP$-cocomplete by \autoref{production}, and each $\bsS^* f$ is $\bsS$\hyp{}cocontinuous by \autoref{invimgiscocont}.

	Now we define the correspondence $\yoth$; on objects we send $A \simeq \bsS G$ into $G$ (by the assumption of the subsection (\autoref{assume}), we can assume that $G$ is $\bsS$-Cauchy complete); of course, we have to check that this is a well-defined assignment: in order to do that, in particular, we must check that $G$ is uniquely determined by $\bsS G$, and that it is $\bsP$-petit in the sense of \autoref{petit_obj}, and $\bsD$-cocomplete, in the sense of \autoref{cocomplete_obj}.
	\begin{enumerate}
		\item Petiteness is ensured by the definition of accessibility, \autoref{yonacc}.
		\item If $\bsS G\simeq \bsS G'$, we can use $\bsS$-Cauchy completeness to build an equivalence $G\simeq G'$ to the effect that $G$ is unique up to equivalence of objects by \autoref{moritaeq}.
		\item $G$ is $\bsD$-cocomplete by \autoref{production}.
	\end{enumerate}
	To define the correspondence $\yoth$ on 1-cells (and on 0-cells as a consequence), we send a right adjoint $f_*: \bsS G'\to \bsS G$ (whose right adjoint is $f^*$) to the 1-cell $u: G\to G'$ induced by the $\bsS$-Cauchy completeness of $G$ (cf. \autoref{caucau}). We need to check that this is $\bsD$\hyp{}cocontinuous. Since this is equivalent to being a pseudomorphism of pseudo algebras, we need to check that the diagram below commutes.
	\[
		\vcenter{\xymatrix{
			\bsD G\ar[r]^{\bsD u}\ar[d]_{l_G} & \bsD G' \ar[d]^{l_{G'}}\\
			G \ar[r]_u & G'
		}}
	\]
	We will show that $\sigma_{G'} \circ u \circ  l_{G}  \cong  \sigma_{G'} \circ l_{G'} \circ \bsD u$, which is enough by the fact that $\sigma_{G'}$ is a pseudo monomorphism (cf. \autoref{pseudomono}). The pseudo equation above is the front face the following elsewhere commutative diagram.
	\[
		\vcenter{\xymatrix@R=6mm@C=6mm{
		& \bsS\bsD G \ar[dd]|\hole\ar[rr]^{\bsS\bsD u}&& \bsS\bsD G' \ar[dd]\\
		\bsD G\ar[ur]^{\sigma_{\bsD G}} \ar[rr]|(.6){\bsD u}\ar[dd]_{l_G}&& \bsD G'\ar[dd]^(.35){l_{G'}} \ar[ur]_{\sigma_{\bsD G'}}\\
		& \bsS G\ar[rr]|\hole^(.35){f^*} && \bsS G' \\
		G \ar[rr]_u \ar[ur]^{\sigma_G}&& G'\ar[ur]_{\sigma_{G'}}
		}}
	\]
	It follows at once that these correspondences define an equivalence of 2-categories.
\end{proof}
\section{Examples}
\begin{eg}[A $0$-example: categories and $\lambda$-presentability]
    In the 2-category of locally small categories, functors, and natural transformations, the `canonical context' where $\bsP$ is the construction of small presheaves and $\bsS$ the completion under $\lambda$-filtered colimits; this yields the classical notions of an accessible and presentable object given in \cite{Adamek1994}.

    The \gu envelope $\bsD_\lambda A$ here is the $\lambda$-colimit completion of $A$; it is easy to see that this \kz{} satisfies the assumptions of \autoref{caucau}, and Gabriel\hyp{}Ulmer duality takes its canonical form (as exposed, for example, in \cite[3.1]{centazzo2002duality}).
\end{eg}
\begin{eg}[The standard example: categories and $\mathbb D$-presentability]
    The former example can be generalised to the case where the context has the form $\upsilon_{\textsc{c},\sD} : \sD\text{-}\bsInd \To \bsP$ and $\sD$ is a `sound' doctrine in the sense of \cite{adamek2002classification}; here, the notion of accessible and presentable object coincide with the notions of $\sD$-\emph{accessible} and \emph{locally $\sD$-presentable} category given in \cite{adamek2002classification}. \cite[76]{centazzo2004generalised} proves that the \gu envelope $\bsD_{\sD} A$ is the $\sD$-colimit completion of $A$; the representation theorem appears in \cite[78]{centazzo2004generalised}. Gabriel\hyp{}Ulmer duality in this context is one of the central result of \cite{centazzo2004generalised}.

    Note that the simplest example of all (the empty doctrine $\sD= \varnothing$) yields a ``trivial'' context, namely $\id_{\bsP} : \bsP \to \bsP$, where accessible objects are precisely presheaf objects.
\end{eg}
\begin{eg}[Posets]
    The collection $\Pos$ of partially ordered classes and monotone class functions becomes a 2-category once $\Pos(P,Q)$ is endowed with the pointwise partial order between such functions. Sending $A\in\Pos$ into $\bsP A \coloneqq \Pos(A, \{0<1\})$ determines a reasonable ``presheaf construction'' on $\Pos$ (this was first noted in \cite{street1978yoneda}). The completion under directed colimits $\bsS$ is well known to the community of poset theory and coincides with the ideal completion $\boldsymbol{Idl}(A)$ of a poset. The locally presentable objects in this Yoneda structure are the \emph{algebraic lattices} in the sense of \cite{porst2011algebraic}, while the accessible objects are ``accessible posets'' (there does not seem to be a name for these categories, but they are just posetal categories that are accessible in $\CAT$); the representation theorem is the content of \cite{porst2011algebraic}. The results in Porst's paper seem to pave the way to a form of Gabriel\hyp{}Ulmer duality; our approach seems to clarify how and why it is so.
\end{eg}
\begin{eg}[Enriched categories]\label{enrich}
    Let $V$ be a locally presentable, monoidal closed category. The notion of an accessible and presentable object in the 2-category of $V$-enriched categories, $V$-enriched functors and $V$-natural transformations, with its natural Yoneda structure having $\bsP A = [A^\op,V]$ has been the subject of a series of works  \cite{borceux1996enriched,borceux1998theory,kelly1982structures}; more in detail, the first two papers establish the theory of accessibility, and the last proves Gabriel\hyp{}Ulmer duality in enriched context. There exists a suitable definition of ``Ind-completion'' worked out in \cite{borceux1998theory}, and \cite[Cor. 3.6]{borceux1996enriched} proves the representation theorem for $V$-enriched categories (a slightly less general version of this result appears as \cite[7.3]{kelly1982structures}). \cite[9.3]{kelly1982structures} proves the existence of a \gu envelope.
\end{eg}
\begin{eg}[Metric spaces and $\one{Ab}$-categories]\label{met-and-ab}
    The former example contains several interesting particular examples:
    \begin{itemize}
        \item if the base of enrichment $V = ([0,\infty],\ge)$ is the monoidal category of non\hyp{}negative real numbers with opposite order, we recover \emph{Lawvere metric spaces} \cite{LawvereFW:metsgl}; the Yoneda structure is given by the ``metric Yoneda embedding'' $X\to [X, V]$. The former example specializes to this context, but $V$-enriched Ind-completion, the \gu envelope and the representation theorem do not seem (to the best of our knowledge) to admit a topological characterisation.
        \item given an additive category $\cA$, regarded as a particular preadditive (=$\one{Ab}$-enriched) category, its Yoneda map shall be $\cA\to [\cA^\op,\one{Ab}]$ in the category of $\one{Ab}$-enriched functors; this is indeed the case; it is easily seen that presentable objects for this KZ context recover the usual theory of locally presentable additive categories.
    \end{itemize}
\end{eg}

\subsection{Homotopical vistas: Derivators}
The original motivation for this paper was to find a good notion of locally presentable and accessible \emph{derivator} which could resemble and possibly put in broader perspective \cite[3.4]{Renaudin2009}, who proposed a definition of \emph{dérivateur de pétite presentation}, and thus a tentative notion of presentable object in the 2-category of prederivators. We aimed to show that this is the `correct' notion of presentability in a 2-category of (pre)derivators.

Grothendieck introduced derivators in order to correct the many shortcomings of triangulated category theory, all rooted in the fact that the embedding of a morphism $f : A \to B$ in a `fibre sequence' $A\xto{f} B\to C\to \Sigma A$ is not functorial. The idea is as simple as follows: instead of looking at a single homotopy category $\text{ho}(\cM)$ of a homotopical category $\cM$, one has to consider the assignment $J\mapsto \text{ho}(\cM^J)$ that sends a small category $J\in \cC\subseteq\Cat$ to the homotopy category of $J$-shaped diagrams valued in $\cM$. If the subcategory $\cC$ of $\Cat$ is large enough to recover information about $\cM$-valued homotopy coherent diagrams, we get a fairly better-behaved object.

Finding enough structure to do formal category theory inside the 2-category of (pre)derivators is a fairly natural idea, having a mass of tangible consequences: such a formal implant would yield a systematic rewriting of derivator theory \cite{groth2013derivators,MR3695365,groth2014tilting,Groth2014d}, that had many applications in the last few years but whose 2-category theory is, to the present day, poorly understood. Moreover, the presence of a context in the sense of our \autoref{def_context} would yield a sufficiently strong form of adjoint functor theorem to be useful in applications: such a fundamental result is absent from the literature to the present day.

So far, the idea seems a practical and sensible proposal. Suddenly, though, several technical and conceptual problems get in the way: building on prior work of Street \cite{street1981conspectus}, it is possible to endow the 2-category of prederivators with a \kz underlying a `variable' Yoneda structure; however, this particular choice of doctrine is of little practical use for homotopy-theoretic driven applications, as it is mostly rooted in the choice of a big enough cardinal $\kappa$ for which all categories $\sD I$ are small, and on the `variable' version of the Yoneda embedding $I\mapsto \Cat(\sD I^\op,{\sf SET})$.

Such categories are always \emph{homotopy} complete and cocomplete. In contrast, the Street Yoneda structure on variable categories described in \cite{street1981conspectus} captures as cocomplete objects those that are \emph{discretely} co/complete: the intersection between these two classes of co/complete objects is of null practical interest.

As true as it may be that the current presentation is just based on an abstract doctrine $\bsP$, it is also true that all interesting examples of a context arise when $\bsP$ has the form of a free cocompletion of some form, and $\bsS$ is a cocompletion under fewer shapes of diagrams. For this reason, our investigation was, until now, based on the following concrete problem: how can one find a doctrine $\bsP$ on the 2-category of prederivators, whose behaviour mimics the presheaf construction?

In order to find it, one shall study prederivators of the form $J\mapsto \text{ho}(\sSet^J)$; this is certainly a natural proposal for what a `cocompletion' operation would have to do, and in fact, it goes in the direction of \cite{Renaudin2009}. The doctrine so determined shall now be able to recognise as $\bsP$-cocomplete objects precisely the objects we expect to be cocomplete in the informal sense. At the same time, it would be nice to have a clear understanding of what a sequentially homotopy cocomplete object should be to make sense of the doctrine $\bsS$ in a homotopy invariant world.

Thus, the present section leaves the reader with more questions than answers and more failed bets than winning hands.
\appendix
\section{Admissible functors of \texorpdfstring{$\Cat$}{Cat}}

This short appendix collects a family of results whose aim is to provide intuition for $\bsP$-admissible functors. We decide not to prove all of them, as we do not use them in the paper, but they can be helpful in the process of understanding $\bsP$-admissible functors.

\begin{rmk}[Characterisation of admissible functors for $\bsP$]\label{char_adm}
  Admissible arrows in $\Cat$ with respect to the small presheaves can be characterized as follows: $f : A \to B$ is admissible if and only if, for each object $b \in B$, the functor $B(f-,b): A^\op \to \Set$ is a small presheaf or, equivalently, if the functor $B(f-,-): B \to \bsP A $ is well defined.

  The reason is that given a functor $K$ in a triangle
  \[
    \vcenter{\xymatrix{
    B\ar[r]^{K} & \bsP A \dltwocell<\omit>{}\\
    & A\ar[u]_{\yon_A}\ar@/^1.5pc/[ul]^f
    }}
  \]
  filled with a 2-cell $\alpha$ assuming there is a functor $B(f,1) : B \to \bsP A$, there is a well-defined map
  \[\Nat(B(f,1), K)\to \Nat(B(f,1)\circ f, K\circ f)\to \Nat(\yon_A, K\circ f) \]
  sending $\alpha$ into $(\alpha * f)\circ \varphi_f$, which is bijective thanks to the fact that each $B(f\firstblank,b)$ is a small presheaf; this allows for the colimit that defines the (pointwise) Kan extension $\Lan_f\yon_A$ to exist, and evidently now $B(f,1)$ has the universal property of $\Lan_f\yon_A$
\end{rmk}

\begin{rmk}[$\bsP$-admissible functors and the Solution Set Condition]\label{solsetcond}
  Freyd introduced the solution set condition as a concrete way to check whether a functor has a chance of being an adjoint. More generally, functors verifying the solution set condition are very tame. $\bsP$-admissibility should be understood as a conceptual understanding of the solution set condition. The two notions are almost equivalent; see the discussion in \cite{ulmer1971adjoint}.
\end{rmk}


\begin{rmk}[$\bsP$-admissible functors are very tractable]
  $\bsP$-admissible functors have not been explored in the literature, but we shall push the statement that they are a very good family of functors to study. In order to do so, let us prove two theorems that show how $\bsP$-admissible functors behave nicely in $\CAT$.
\end{rmk}

\begin{prop}\label{admis_cond}
  Let $f: A \to B$ be functor between locally small categories. Then, the following are equivalent.
  \begin{enumerate}
    \item f is $\bsP$ admissible;
    \item for every functor $g: A \to C$, where $C$ is a small-cocomplete category, the Kan extension $\lan_f g$ exists and is point-wise.
  \end{enumerate}
\end{prop}

\begin{cor}
  Let $f: A \to B$  functor between (small)-cocomplete categories. Then, the following are equivalent.
  \begin{enumerate}
    \item $f$ is a left adjoint;
    \item $f$ is cocontinuous and $\bsP$-admissible.
  \end{enumerate}
\end{cor}

\newpage
\bibliography{allofthem}{}
\bibliographystyle{amsalpha}
\hrulefill
\end{document}